\documentclass[12pt]{amsart}

\usepackage{amssymb,amsmath,amsthm,hyperref}
\hypersetup{
pdftitle={Research Paper},
pdfauthor={Abbas Nasrollah Nejad},
pdfsubject={The Module of Valabrega-Valla of the Jacobian Ideal of Points },
pdfcreator={Abbas Nasrollah Nejad},
colorlinks,
linkcolor=blue,
citecolor=magenta,
anchorcolor=red,
bookmarksopen,
urlcolor=red,
filecolor=red,
pdfpagetransition={Wipe}
}
\usepackage{graphicx}

\theoremstyle{plain}
\newtheorem{Theorem}{Theorem}[section]

\newtheorem{Proposition}[Theorem]{Proposition}
\newtheorem{Conjecure}[Theorem]{Conjecture}

\theoremstyle{definition}
\newtheorem{Definition}[Theorem]{Definition}

\theoremstyle{Remark}
\newtheorem{Remark}{Remark}
\newcommand{\rar}{\rightarrow}
\newcommand{\lar}{\longrightarrow}
\newcommand{\surjects}{\twoheadrightarrow}

\def\fm{{\mathfrak m}}

\def\hht{{\rm ht}\,}

\def\cl#1{{\mathcal #1}}

\def\VaVa{{\cl V}\kern-5pt {\cl V}}

\def\pp{{\mathbb P}}

\def\X{{{\mathbb X}}}

\begin{document}

\title[ Valabrega-Valla of the Jacobian ideal of points ]{ The module of Valabrega-Valla of the Jacobian ideal of points in projective  plane}
\author[A. Nasrollah Nejad \,  Z. Shahidi ]{  abbas Nasrollah nejad  \,  Zahra Shahidi}
\address{ Department of Mathematics, Institute for Advanced Studies in Basic Sciences (IASBS), Zanjan, Iran.}
\email{abbasnn@iasbs.ac.ir \,\,    z.shahidi@iasbs.ac.ir}
\begin{abstract}
The module of Valabrega-Valla of the Jacobian ideal of a  reduced projective variety $V$ is the torsion of the Aluffi algebra. One considers the problem of its vanishing in the case of where $V$ is a reduced set of points in the projective plane. It is shown that the module is nonzero for several cases of a special configuration class therein -- called $(s-r)$-{fold collinear configuration}. A complete classification of types is given for $5$ and $6$ points in regard to this problem. %For a set of five and six points, we characterize when the  module of Valabrega-Valla vanishes in  geometry of points. 
\end{abstract}
\maketitle
\section*{Introduction}
Let $V\subset \pp^d=\pp^d_k$ denote a reduced projective variety over a perfect field $k$, with homogeneous coordinate ring $A=R/J$, where $R=k[x_0,\ldots,x_d]$ and $J$ is the homogeneous defining ideal of $V$ in its embedding. 
%For the sake of the discussion, we assume that $k$ is  
By definition, the {\em Jacobian ideal} of the $k$-algebra $A$ is the Fitting ideal of order $d+1-(d-\dim V)=\dim V+1$ of the module of K\"ahler $k$-differentials $\Omega_{A/k}$ of $A$.
Since $\Omega_{A/k}$ depends only on $A$, and not on any particular presentation of $A$, and so do the Fitting ideals of $\Omega_{A/k}$, then this notion is dependent only upon the given projective embedding of $V$.
This is as much invariance one can dispose of.

Now, since $\Omega_{A/k}$ admits a module presentation as the cokernel of the transposed Jacobian matrix $\Theta$ of a set of generators of $J$, modulo $J$, then the Jacobian ideal of $A$ is the ideal of $A$ generated by the $(d-\dim V)$-minors of this Jacobian matrix.
By abuse, one often takes the ideal of $(d-\dim V)$-minors back in the polynomial ring $R$ as the Jacobian ideal of the ideal $J$. If this malpractice is followed then, in order to keep the invariant properties of the Jacobian ideal of $A$, one should always take those minors summed to $J$  -- that is, the ideal $I:= (J,I_{d-\dim V}(\Theta))\subset R$.

%If needed to focus solely on the determinantal ideal $I_{d-\dim V}(\Theta)$ then we choose to call the latter the {\em critical ideal} of $A$, if only to honor the classical terminology ``critical locus''.

In this paper we focus on the problem as to when the so-called Valabrega--Valla module vanishes for the pair $J\subset I$, where $J$ is a homogeneous ideal in a polynomial ring over a field and $I$, denotes its Jacobian ideal.
For that a basic principle was established in \cite[Example 2.19]{AA} and successfully applied to a collection of classical geometric situations, such as the rational normal curve (codimension $d-1$) and certain Segre and Veronese embeddings.
A little later, this problem was applied in \cite{AR} for special classes of linear determinantal ideals (rational normal scrolls and alike).

In \cite{AAR} the problem was circumscribed to the case of  a set of reduced points when the latter admits a minimal set of generators in degree $2$. In \cite[Theorem 1.2]{AZR} a characterization of the vanishing of the module of Valabrega-Valla was given in terms of the first syzygy module of the form ideal $J^*$ in the associated graded ring of $I$. In \cite{AZ} a vast extension of these matters was taken all the way to the environment of modules, bringing up some hard features of Cohen--Macaulay and Gorenstein algebras naturally arising from Rees algebras of modules.

Given an inclusion of ideals $J\subset I$, the {\em module of Valabrega--Valla} is defined as the graded module 
\begin{equation*}\label{vava}
\VaVa_{J\subset I}:=\bigoplus_{t\geq 2} \frac{J\cap
	I^t}{JI^{t-1}}.
\end{equation*}
As such, it has appeared elsewhere  in a different context (see
\cite{VaVa}, also \cite[5.1]{Wolmbook1}).
As it turns out, provided $I$ has a regular element module $J$, the Valabrega--Valla module is the $A$-torsion of the so-called embedded Aluffi algebra of $I/J$, hence its interest for the geometric purpose~\cite{Thesis1}.
Dealing directly with the  Valabrega--Valla module makes the structure of the Aluffi algebra itself sort of invisible. On the bright side, in the case where $\mathfrak a=I/J$ is the Jacobian ideal of $A=R/J$, the heavy work is transferred to understanding its nature.
Besides, for some mysterious reason, the existence of non-trivial torsion is often delivered at the level of degree $2$ of ${\cl V}\kern-5pt {\cl V}_{J\subset I}$.

The Jacobian ideal $I/J\subset A:=R/J$ will be said to be {\em ${\cl V}\kern-5pt {\cl V}$-torsion free} if ${\cl V}\kern-5pt {\cl V}_{J\subset I}=\{0\}$, i.e., if $J\cap I^t=JI^{t-1}$ for all $t\geq 2$.

The main goal of this work is to understand the nature of the $\VaVa$-torsion freeness of the Jacobian ideal $I$ of an ideal $J\subset R:=k[x_0,\ldots, x_n]$ of a finite set of distinct  points in projective space. Note that, since $R/J$ is Cohen--Macaulay, $R/J$ is reduced if and only if $\hht (I)\geq \hht (J)+1$, which is equivalent to saying that $J$ is an ideal of reduced points if and only if $I$ contains a regular element modulo $J$ or, still, to assert that $I$ is an $\fm$-primary ideal, where $\fm=(x_0,\ldots, x_n)$.
As such, the only algebraic subtlety away from the geometric data is the number and the degrees of a set of minimal generators of the $\fm$-primary ideal $I$. Indeed, this emphasis will prevail throughout the paper. 

The results and examples in \cite{AAR} show that the answer to this problem may not have a general shape, even when one assumes that the defining ideal is generated by quadrics -- a condition that is only guaranteed when the number of points in $\pp^n$ is at most $2n$. In $\pp^2$ this forces the number of points to be at most $4$, not a very bright situation. The idea of this work is to go beyond, by rather focusing on special configuration of points in  the projective plane $\pp^2$, without imposing constraints on the number of points. However, we do consider the `next' cases: $5$ and $6$ points are completely classified. 

The reason to look at reduced plane points is that, along with the results of \cite{AA} on hypersurfaces, it would give a reasonable picture of the Jacobian torsion for the subvarietis of $\pp^2$.
The next difficult step would be the case of curves in $\pp^3$, a problem that we may tackle in the near future.
It is the authors expectation that new suitable techniques be brought in to get a bird view of the problem in all cases.

The outline of the paper is as follows. 

In section~\ref*{P1}, we give a quick view of the main characters for arbitrary ideals in a Noetherian ring (for further details we refer to \cite{AA}). Then proceed to the case of the defining ideal of a reduced set of points in projective space. In this regard, Proposition~\ref{ATFN} gives a criterion to check $\VaVa$-torsion freeness of the Jacobian ideal of a finite set of points in a projective space. 

In section~\ref{P2},  we consider the following configuration of points in $\pp^2$: given integers $s\geq 4$ and $0\leq r\leq s-3$, will say that a finite set of points $\X$ is an $(s-r)$-{\em fold collinear configuration} when $(s-r)$ among them lie on a straight line. We prove that if $r=0,1$, then the Jacobian ideal of the defining ideal $I(\X)$ of $\X$ is not $\VaVa$-torsion free (Theorem~\ref{collinear-colinear-1}). In the case that $r=2,\, s\geq 8$ or $r=3,\, s\geq 9$, we show that the Jacobian ideal of $\X$ is not $\VaVa$-torsion free. 
Clearly, all theses cases  have a simple geometry, which however does not seem to yield automatically the `asymptotic' behavior of the problem in consideration.  

In section~\ref{4-5-6points}  we consider the defining ideal $I(\X)$ of a set $\X$ of $5$  and $6$ points in $\pp^2$. In the case of $5$ distinct points it is known that  there are $5$ mutually distinct configurations. We show that the Jacobian ideal of $I(\X)$ is $\VaVa$-torsion-free exactly in the following configurations:
\begin{itemize}
	\item $\X$ is in general linear position.
	\item $\X$ is a $3$-fold collinear configuration.
	\item $\X$ is a $3$-fold collinear configuration such that the straight line through the remaining two points intersect in a point of $\X$. 
\end{itemize}
Finally, let $\X\subset \pp^2$ be a set of six points.
For $6$ points, there are eleven distinct configurations (see Figure ~\ref{6points Configuration}).
We show that the Jacobian ideal of $I(\X)$ is $\VaVa$-torsion-free if and only if $\X$ is in one of the following  configurations:
\begin{itemize}
	\item $\X$ is in  general linear position.
	\item $\X$ is  a $4$-fold collinear configuration.
	\item $\mathbb{X}$ is a $4$-fold  collinear configuration such that  the straight line through the remaining two points intersect in a point of $\X$. 
	\item $\X$ is  a $3$-fold collinear configuration such that  the remaining three points are in general linear position. 
	\item $\X$ is a $3$-fold collinear configuration such that  the straight line through two of  the  remaining  points  intersect in a point of $\X$.
	\item $\X$ is a $3$-fold collinear configuration such that  two straight lines through of  the  remaining  points  intersect in  points of $\X$.
	\item $\X$ is a $3$-fold collinear configuration such that the  three  straight lines through of  the remaining  points  intersect in  points of $\X$.
\end{itemize}
\section*{Acknowledgment}
{The second author thanks the Instituto de Ci\^encias Matem\'aticas e da Computa\c c\~ao (ICMC, S\~ao Carlos, Brazil) and the Department of Mathematics of the Federal University of Sergipe (UFS, Brazil) for providing a suitable atmosphere for her stay in the frame of a sabbatical leave. Both authors thank  Zaqueu Ramos and Aron Simis for insightful discussions on 	the preliminary versions of this paper. Simis, in particular, has been helpful in suggesting a couple of  improvements in the style of some proofs}

\section{The module of Valabrega-Valla of points}\label{P1}
Given a Noetherian ring $R$ and ideals $J\subset I$ of $R$, the module of Valabrega-Valla is defined as graded module
\begin{equation*}\label{vava}
\VaVa_{J\subset I}:=\bigoplus_{t\geq 2} \frac{J\cap
	I^t}{JI^{t-1}}.
\end{equation*} 
There is a natural surjective $R/J$-algebra homomorphism from the Aluffi algebra to the Rees algebra of $I/J$
\[\mathcal{A}_{R/J}(I/J)\simeq\bigoplus_{t\geq 0} I^t/JI^{t-1}\surjects \mathcal{R}_{R/J}(I/J)\simeq\bigoplus I^t/J\cap I^t. \]
The Aluffi algebra is an algebraic version of characteristic cycles in intersection theory~\cite{Thesis1}. The connection of module of Valabrega-Valla  with the Aluffi algebra is as follows:
\begin{Proposition}{\rm (\cite[Proposition 2.5]{AA})}\label{Aluffi_torsion}
	If $I$ has a regular element module $J$, then the module of  Valabrega--Valla  is the $R/J$-torsion of the  Aluffi algebra of $I/J$.
\end{Proposition}
The vanishing of $\VaVa_{J\subset I}$ has close relation with the theory of $I$-standard base ( in the sense of Hironaka), Artin Rees number and relation type number (See~\cite{FPV1} and \cite{AR}). 

\smallskip

Let $\X:=\{p_1,\ldots,p_s\}$ be a set of $s$  distinct points in $\pp^n:=\pp_k^n$, where $k$ is an algebraically closed field of characteristic zero and $n\geq 2$. The defining ideal of $\X$ is the ideal $I(\X)=\bigcap_{i=1}^sI(p_i)\subseteq R=k[x_0,\ldots,x_n ]$ where $I(p_i)$ is the prime ideal generated by $n$ linear forms. Note that the coordinate ring  $R/I(\X)$ of $\X$ is a reduced ring of dimension one and hence is a Cohen-Macaulay ring. The multiplicity of $R/I(\X)$ is the number of points  in $\X$ \cite[Corollary 3.10]{David}. Since $R/I(\X)$ is locally regular on the punctured spectrum, by the Jacobian criterion it translates into the property that the Jacobian ideal  $I=(I(\X),I_n(\Theta))$ is $\fm$-primary, where  $I_n(\Theta)$ is the ideal of $n$-minors of the  Jacobian matrix  $\Theta$ of $I(\X)$ and $\fm$ denotes the maximal irrelevant ideal of the polynomial ring $R$. In other words, there is a suitable power $\fm^l$ that lands into $I$. Therefore, if the defining ideal of $\X$ is minimally generated in single degree $m\geq 2$ and $\fm^{n(m-1)}\subset I$, then the Valabrega-Valla module vanishes~\cite{AAR}. 

Let us introduce the following terminology:

\begin{Definition}\rm
	The Jacobian ideal $I/J\subset A:=R/J$ will be said to be {\em ${\cl V}\kern-5pt {\cl V}$-torsion free} if ${\cl V}\kern-5pt {\cl V}_{J\subset I}=\{0\}$, i.e., if $J\cap I^t=JI^{t-1}$ for all $t\geq 2$.
\end{Definition}

Given a ring $A$ and an ideal $\mathfrak{a}=(a_1,\ldots,a_n)\subseteq A$, one lets $\varphi: A[T_1,\ldots,T_n]\surjects \mathcal{R}_A(\mathfrak{a})=A[\mathfrak{a}t]$ be the graded map sending $T_i$ to $a_it$. The relation type number of $\mathfrak{a}$ is the largest degree of any minimal system of homogeneous generators of the kernel $\varphi$. Since the isomorphism $A[T_1,\ldots,T_n ]/\mathrm{Ker}\varphi\simeq \mathcal{R}_A(\mathfrak{a})$ is graded, an application of the Schanuel lemma to the graded pieces shows that the notion is independent of the set of generators of $\mathfrak{a}$. 

A key result in the case of points reads like this:  
\begin{Proposition}\label{ATFN}
	Let $\X$ be a finite set of $s$ points in $\pp^n$ and $I$ stands for the Jacobian ideal of $I(\X)$.  Assume that $ I(\X) \cap I^t=I(\X)I^{t-1}$ for every $ 2 \leq t \leq s $. Then the Jacobian ideal of  $I(\X)$ is $\VaVa$-torsion-free. 
\end{Proposition}
\begin{proof}
	By \cite[Corollary 2.17]{AA} it suffices to prove that the relation type number of  $I/I(\X)$ is at most $s$.
	%Since $\X$ is smooth, the Jacobian ideal $I/I(\X)$ is $\fm/I(\X)$-primary where $\fm=(x_0,\ldots,x_n)$. By \cite[Lemma 6.3]{FPV}, it is enough to show that the relation type number of  $I/I(\X)$ is at most $s$. 
	But as $R/I(\X) $ is a $1$-dimensional Cohen Macaulay graded ring, then \cite[Lemma 6.3]{FPV} implies that the relation type number of $ I/I(\X)$ is bounded by the multiplicity of $ R/I(\X)$ which is the number of points in $ \X$.
\end{proof}
\begin{Remark}\rm
	 (a) By Proposition ~\ref{ATFN}, to check the $\VaVa$-torsion-freeness it is enough to check for $ t \leq s$, where $ s$ is the number of points. This bound is not sharp in general -- see, e.g., Proposition~\ref{5-points}. 
	
	(b) Let $ \X$ be a set of $s$ collinear points with $ s\geq 4$ in $\pp^2$. Then the defining ideal of $\X$ is $ I(\X)=(z,f(x,y))$, where $ f(x,y)$ is a reduced product of $s$ linear forms in $x,y$.  Then the relation type of $I/I(\X)$ on $R/I(\X)$ is the degree of the dual form of $f(x,y)$, which is itself (for the details see Proposition~\ref{collinear-colinear-1}). Thus, the relation type is  $s$ (maximal possible), but $ I(\X)\cap I^2\neq  I(\X)I$ as can be readily obtained by considering the element $y^{2s-4}f$.
\end{Remark}

\section{$(s-r)$-fold collinear configurations in $\pp^2$}\label{P2}

We introduce the following  configuration of points in $\pp^2$. 
\begin{Definition}\rm\label{(s-r)configuration}
	Let $s\geq 4$ and $0\leq r\leq s-3$ be integers.
	An {\em $(s-r)$-fold collinear configuration} is a finite set of $s\geq 4$ plane points such that exactly $(s-r)$ of its points lie on a straight line.  
\end{Definition} 
\begin{Theorem}\label{collinear-colinear-1}
	Let $\mathbb{X}\subset \pp^2$ be an $(s-r)$-fold collinear configuration. If $r=0,1$, then the Jacobian ideal of $I(\mathbb{X})$ is not ${\cl V}\kern-5pt {\cl V}$-torsion free. 
\end{Theorem}
\begin{proof}
Suppose that $r=0$, i.e., all points are collinear.  Say, the points lie on the line $\{z=0\}$. Then the ideal of a point $p_i (1\leq i\leq s)$ has the form  $I(p_i):=(z,\ell_{i}(x,y))$, for some $k$-linear form $ \ell_{i}(x,y)\in k[x,y]$.
Moreover, these linear forms are independent.
Then, it is an easy exercise to get
$$I(\X)=\bigcap_{i=1}^s I(p_i)=(z,f(x,y)),$$
where $f=\prod_{i=1}^{s}\ell_{i}(x,y)$. 

By using Euler formula, the Jacobian ideal of $I(\mathbb{X})$ is $I=(z,\partial f/\partial x, \partial f/\partial y)$.  
Now, since $R/I(\X)\simeq k[x,y]/(f)$ and $I/I(\X)\simeq (\partial f/\partial x, \partial f/\partial y)/(f)$, one can argue with these simplified structures.
Since the ideal $(\partial f/\partial x, \partial f/\partial y)$ is a complete intersection, hence of linear type, it follows that the Aluffi algebra of $(\partial f/\partial x, \partial f/\partial y)/(f)$ is isomorphic to its symmetric algebra (see~\cite{AA}), whose ideal of relations contains no form of bidegree $(0,d)$, with $d\geq 1$. 
On the other hand, the defining relations of its Rees algebra contains the equation of the dual curve to $f$, which is $f$ itself read in the relational variables.
This gives a relation of bidegree $(0,s)$, $s\geq 4$. 
By Proposition~\ref{Aluffi_torsion}, this relation gives a non-trivial torsion.

\medskip

Now let $\mathbb{X}$ be an $(s-1)$-fold collinear configuration. As before, assume that the points $p_1,\ldots,p_{s-1}$ lie on  $\{z=0\}$, while $p_s\notin \{z=0\}$.  Let $l_{i,s}\in k[x,y]$ denote the linear form defining the unique straight line through  $p_i$ and $p_s$, for $i=1,\ldots,s-1$. 
Then a simple calculation yields
\begin{eqnarray}\label{simple_calc}\nonumber
I(\X)&=&\left(\bigcap_{i=1}^{s-1} I(p_i)\right)\cap I(p_s)=\left(\bigcap_{i=1}^{s-1} (z,l_{i,s})\right)\cap (l_{1,s}, l_{2,s})=(z, \prod_{i=1}^{s-1}l_{i,s})\cap (l_{1,s}, l_{2,s})\\
&=& \left((z)\cap (l_{1,s}, l_{2,s}), \prod_{i=1}^{s-1}l_{i,s}\right)=\left(zl_{1,s}, zl_{2,s}, \prod_{i=1}^{s-1}l_{i,s}\right)
\end{eqnarray}
since $\{z,l_{1,s}, l_{2,s}\}$ is a regular sequence  both linear forms $l_{1,s}, l_{2,s}$  are factors of $\prod_{i=1}^{s-1}l_{i,s}$.

Next, writing $p_i=[a_i:1:0]$ for $i=1,\ldots, s-1$ and $p_s:=[0,0,1]$, hence $l_{i,s}=x-a_iy$. Setting $f:=\prod_{i=1}^{s-1}l_{i,s}$, 
the Jacobian matrix of $I(\mathbb{X})$ is 
\[\Theta= \begin{bmatrix}
z & -a_1z & x-a_1y\\
z & -a_2z & x-a_2y\\
f_x & f_y & 0\\
\end{bmatrix}.\]

The $2\times 2$-minors fixing the last row yield the following subset $\{xf_x,yf_x,xf_y,yf_y\}\subset k[x,y]$ of a set of minimal generators of the Jacobian ideal $I$. Set $\mathfrak{a}:=(xf_x,yf_x,xf_y,yf_y)$, as an ideal of the ring $k[x,y]$. 

{\sc Claim.} $(f)\cap \mathfrak{a}^2\neq (f)\mathfrak{a}$.

To see this, note that, as an ideal of codimension $2$ in $k[x,y]$, $\mathfrak{a}$ is $(x,y)$-primary, hence perfect. Thus, its matrix of syzygies is a $4\times 3$ matrix containing the obvious columns
\begin{equation}\label{linear_syzygies}
\begin{array}{cc}
-y&0\\
x&0\\
0&-y\\
0&x
\end{array}
\end{equation}
Counting degrees and letting $S=k[x,y]$,  the minimal free resolution of $S/\mathfrak{a}$ over $S$ has the form
\begin{equation}\label{Res I1}
0\rar S^2(-s)\bigoplus S(-(2s-4))\rar S^4(-(s-1))\rar S\rar S/\mathfrak{a}\rar 0.
\end{equation}
Therefore, the minimal free resolution of $S/(f)\mathfrak{a}$ over $S$ has the form
\begin{equation}\label{Res Ibis}
0\rar S^2(-(2s-1))\bigoplus S(-(3s-5))\rar S^4(-2(s-1))\rar S\rar S/(f)\mathfrak{a}\rar 0.
\end{equation} 

Now consider the ideal $\mathfrak{a}^2$. By the same token, it is a perfect ideal of codimension $2$. Moreover, since $f_x$ and $f_y$ have no common factors, it will turn out that  $\mathfrak{a}^2$ is minimally generated by  $9$ forms of degree $2(s-1)$.
It follows that its syzygy matrix is $9\times 8$.
Since the linear syzygies in (\ref{linear_syzygies}) generate $6$ independent linear syzygies of  $\mathfrak{a}^2$, it follows that the remaining two syzygies have the same  degree $3(s-2)$.
It follows that the minimal free $S$-resolution of $S/ \mathfrak{a}^2$ has the form
\begin{equation}\label{Res I2}
0 \rar  S^6(-(2s-1))\bigoplus S^2(-(3(s-2)))\rar S^9(-(2(s-1)))
\rar  S\rar S/\mathfrak{a}^2\rar 0
\end{equation}
Now, consider the ideal $(f,\mathfrak{a}^2)$.
Drawing upon the Euler relation, it is easy to see that this ideal is minimally generated by $f$ and $5$ more among the original generators of $\mathfrak{a}^2$.
Moreover, it inherits $4$ linear syzygies among those of $\mathfrak{a}^2$.
Therefore, its minimal free $S$-resolution has the shape
\begin{eqnarray}\label{Res I4}
0&\rar & S(-(2s+1))\oplus S^4(-(2s-1))\rar S(-(s-1))\oplus S^5(-(2s-2))\\ \nonumber
&\rar & S\rar S/(f,\mathfrak{a}^2)\rar 0.
\end{eqnarray}
Finally, take the exact squence 
\begin{equation*}
0\rar \dfrac{S}{(f)\cap \mathfrak{a}^2}\rar \dfrac{S}{(f)}\bigoplus \dfrac{S}{\mathfrak{a}^2}\rar \dfrac{S}{(f,\mathfrak{a}^2)}\rar 0.
\end{equation*}
Applying to this exact sequence the information gathered in (\ref{Res I2}),(\ref{Res I4}), we find the Hilbert series of the left-most term:
\[\mathrm{HS}\left(\frac{R}{(f)\cap \mathfrak{a}^2}, t\right)=\dfrac{1+t+t^2+\cdots+t^{2s-3}-3t^{2s-2}-t^{2s-1}-t^{2s}-\cdots-2t^{3s-7}}{(1-t)}. \]
Using (\ref{Res Ibis}), one has
\[\mathrm{HS}\left(\frac{R}{(f)\mathfrak{a}}, t\right)= \dfrac{1+t+t^2+\cdots+t^{2s-3}-3t^{2s-2}-t^{2s-1}-t^{2s}-\cdots-t^{3s-7}-t^{3s-6}}{(1-t)}.\]
As the two  Hilbert series are different, it follows that $(f)\cap \mathfrak{a}^2\neq (f)\mathfrak{a}$. This completes the proof of the claim.

The Jacobian ideal $I$  is minimally generated by the monomials  $xz,yz,z^2$ and $\mathfrak{a}$. By the above claim, there exists a polynomial $g\in (f)\cap \mathfrak{a}^2\subset  S$ such that $g\neq (f)\mathfrak{a}$. Since the ideal $I(\X)I$ is generated by  polynomials which contain the variable $z$ and $f\mathfrak{a}$, it follows that  the polynomial $g$ does not belong to $I(\X)I$. This proves that the Jacobian ideal of $I(\X)$ is not ${\cl V}\kern-5pt {\cl V}$-torsion free. 	
\end{proof}

\smallskip

The configuration studied so far requires that $r\leq s-3$, hence $r\leq 1$ for $s=4$.
However, for $s\geq 5$, $r=2$ is a kosher value.
We now digress on a slight degeneration of such a configuration.

\begin{Theorem}\label{s-2points-collinear}
	Let $\mathbb{X}\subset \pp^2$ be an $(s-2)$-fold collinear configuration of  $s\geq 8$ distinct points.  
	Then  the Jacobian ideal  of $I(\X)$ is  not $\VaVa$-torsion-free. 
\end{Theorem} 
\begin{proof}
There are two sub-cases of this configuration, according to which the collinearity line and the straight line through the remaining two points intersect in $\X$ or off $\X$. Let $\X_{\mathrm{in}}$ and $\X_{\mathrm{off}}$ stand for these configurations of points, respectively. 
 
We compute a set of minimal generators of $I(\X_{\mathrm{in}})$ and $I(\X_{\mathrm{off}})$ and this will work fine for any $s\geq 5$.
Take $p_1,\ldots,p_{s-2}$ to lie on the straight line $\{z=0\}$. Let  $p_{s-1}$ and $p_s$ the remaining points off the collinear line. Let $l_{i,s-1}$ and $l_{i,s}$ denote, respectively, the linear forms defining the straight lines through $p_i$ and $p_{s-1}$ and through $p_i$ and $p_s$, for $i=1,\ldots,s-2$. 
Finally, let $l$ denote the linear form defining the straight line through $p_{s-1}$ and $p_s$. 

A simple calculation, based on the same elementary principles as for (\ref{simple_calc}), yields
\begin{eqnarray*}
	I({\X_{\mathrm{in}}})&=&\left( \bigcap_{i=1}^{s-3}(z,l_{i,s})\right)\cap (z,l)\cap (l_{1,s},l)\cap (l_{1,s-1},l)= \left(z,\prod_{i=1}^{s-3}l_{i,s}\right)\cap \left(zl_{1,s}l_{1,s-1},l\right)\\
	&=& \left(zl_{1,s}l_{1,s-1}, l\cap (z,\prod_{i=1}^{s-3}l_{i,s})\right)=\left(zl,\, zl_{1,s}l_{1,s-1},\,l\prod_{i=1}^{s-3}l_{i,s}\right).
\end{eqnarray*}
Now setting 
$$J:=(zl,\, zl_{1,(s-1)}l_{(s-2),s},\,l_{(s-2),s}\prod_{i=1}^{s-3}l_{i,(s-1)}).$$ Clearly $J\subseteq I(\X_{\mathrm{off}})$. By construction, $J$ is generated by $2\times 2$ minors of the $2\times 3$ Hilbert-Burch matrix 
\[
\begin{bmatrix}
l_{1,(s-1)}l_{(s-2),s}& 0\\
l& \prod_{i=2}^{s-3}l_{i,(s-1)}\\
0&z 
\end{bmatrix}.
\]
Hence the minimal free $R$-resolution of $R/J$ has the form 
\[0\rar R(-4)\oplus R(-(s-1))\rar R(-2)\oplus R(-3)\oplus R(-(s-2))\rar R\rar R/J\rar 0.      \]
It follows that the Hilbert series of $R/J$ is 
\[\mathrm{HS}(R/J,t)=\dfrac{1-t^2-t^3+t^4-t^{s-2}+t^{s-1}}{(1-t)^3}=\dfrac{1+2t+2t^2+t^3+\ldots+t^{s-3}}{1-t}. \]
Write $\X_{\mathrm{off}}:=\X_1\cup \X_2 $, where $\X_1$ is the set of $s'=s-1$ points in $(s'-1)$-fold collinear configuration and $\X_2=\{p_s\}$. One has the short exact sequence 
\begin{equation*}\label{ExactSeq}
0\rar R/I(\X_{\mathrm{off}})\rar R/I(\X_1)\bigoplus R/I(\X_2)\rar R/(I(\X_1),I(\X_2))\rar 0.
\end{equation*}
Direct inspection gives
\[R/(I(\X_1),I(\X_2))\simeq k[z]/(z^2)\quad, \quad R/I(\X_2)\simeq k[z] .\] 
Using the above exact sequence and Theorem~\ref{collinear-colinear-1}, we get 
\begin{eqnarray}
\nonumber \mathrm{HS}(R/I(\X_{\mathrm{off}}))&=&\dfrac{1+2t+t^2+t^3+\ldots+t^{s-3}}{1-t}+\dfrac{1}{1-t}-(1+t)\\
\nonumber &=&\dfrac{1+2t+2t^2+t^3+\ldots+t^{s-3}}{1-t},
\end{eqnarray}
which proves that $J=I(\X_{\mathrm{off}})$.

Now we prove $\VaVa_{I(\X_{\mathrm{off}})\subset I}\neq0$. Similar argument apply to the configuration $\X_{\mathrm{in}}$. 
Writing $p_i=[1:a_i,0]$  with $a_i\neq 0,1$ for $i=1,\ldots,s-2$, $p_{s-1}=[0:0:1]$ and $p_s=[1:1:1]$. It follows that the ideal $I(\X_{\mathrm{off}})$ is generated by 
\[q:=xz-yz\quad,\quad c:=zy^2-yz^2  \quad,\quad f:=\sum_{i=0}^{s-2}g_ix^{s-2-i}y^{i}+yz^{s-3}, \]
where each $g_i$ is a certain polynomial expression of the $a_i$'s. 
The Jacobian matrix of $I(\X_{\mathrm{off}})$ is
\[\Theta=
\begin{bmatrix}
z & -z & x-y\\
0 & 2yz-z^2 & y^2-2yz\\
f_x & f_y & f_z\\
\end{bmatrix}
,\]
where 
\[ f_x  =  \sum_{i=1}^{s-2}(s-2-i)g_ix^{s-3-i}y^{i},\,f_y  =  \sum_{i=1}^{s-2}(i)g_ix^{s-2-i}y^{i-1}+z^{s-3}, \, f_z  =  (s-3)yz^{s-4}.  \]

The ideal generated by the  $2\times 2$ minors fixing the first two rows  and the generators $q$ and $c$ yield the subset $\{xz-yz,yz^2,zy^2,z^3\}$ of a set of minimal generators of the Jacobian ideal $I$. Setting $\mathfrak{a}:=(xz-yz,yz^2,zy^2,z^3)$.   
One can verify  that the polynomials $zf_x,\, zf_y,\, zf_z$  belong to the ideal $\mathfrak{a}$ and  the Jacobian ideal is minimally generated by
\[I=(xz-yz\ ,\ y^2z \ ,\ yz^2,\ z^3,\ G\ ,\ (y-x)f_x\ ,\ (y-x)H\ ,\ y^2f_x) ,\]
where $G=\sum_{i=0}^{s-2}g_ix^{s-2-i}y^{i}$ and $H=\sum_{i=1}^{s-2}(i)g_ix^{s-2-i}y^{i-1}$. 
Set 
$$\mathfrak{b}:=(G,\,(y-x)f_x,\,(y-x)H,\,y^2f_x),$$ 
a zero dimensional homogeneous ideal.
Then $\mathfrak{b}\cap k[y]\neq 0$. The following relation for suitable $ \alpha_i, \beta_i, \gamma_i$ -- certain polynomials like expression in the $a_i$'s -- shows that $ y^{2s-7} \in \mathfrak{b} $ and $ 2s-7$ is  minimum with this property. $$ y^{2s-7}= \lambda_1G+\lambda_2(y-x)H+\lambda_3y^2f_x \in \mathfrak{b},$$ where 
\begin{eqnarray}
\nonumber \lambda_1 & = & \alpha_1x^{s-5}+\alpha_2x^{s-6}y+\ldots+ \alpha_{s-5}xy^{s-6}+\alpha_{s-4}y^{s-5},\\
\nonumber \lambda_2 & = & \beta_1x^{s-6}y+\beta_2x^{s-5}y^2+\ldots+\beta_{s-6}xy^{s-6}+\beta_{s-5}y^{s-5},\\
\nonumber \lambda_3 & = & \gamma_1x^{s-6}+\gamma_2x^{s-7}y+\ldots+\gamma_{s-6}xy^{s-7}+\gamma_{s-5}y^{s-6}.\ \    
\end{eqnarray}
Now consider the polynomial $E:=y^{2s-8}G-yz^{2s-3}$. Clearly, $E$ dose not belong to $I(\X_{\mathrm{off}})I$. We show that $E\in I(\X_{\mathrm{off}})\cap I^2$, which proves that $\VaVa\neq 0$ in degree $2$.  We have 
$$y^{2s-8}G-yz^{2s-3}=y^{2s-8}f+ (y^{2s-9}z^{s-4}+y^{2s-10}z^{s-3}+ \ldots + yz^{s+2}+z^{s+3})c \in I(\X), $$ 
and
\begin{eqnarray}
\nonumber y^{2s-8}G &=& T_1G^2+T_2((y-x)f_xG)+T_3((y-x)HG)+
+T_4((y-x)^2f_x^2)\\
\nonumber 	 &+&T_5((y-x)^2f
+ T_6((y-x)^2H^2)+\ell_{1}(y^2f_xG)	+\ell_{2}((y-x)y^2f_x^2)\\
\nonumber &+ &\ell_{3}((y-x)Hy^2f_x)+y^4f_x^2   \in \mathfrak{b}^2\subseteq I^2, 
\end{eqnarray}
where
\begin{eqnarray}
\nonumber T_{i} & = & \gamma_{i1}x^{s-6}+\gamma_{i2}x^{s-7}y+\ldots+ \gamma_{i(s-6)}xy^{s-7}+\gamma_{i(s-5)}y^{s-6}, \  \  1 \leq i \leq 6 \\
\nonumber \ell_{i} & = & \beta_{i1}x^{s-7}+\beta_{i2}x^{s-8}y+\ldots+\beta_{i(s-7)}xy^{s-8}+\beta_{i(s-8)}y^{s-7}, \  \   1 \leq i \leq 3,
\end{eqnarray}
with $\beta_{ij}, \gamma_{ij}\in k$.  Finally,  $ yz^{2s-3}= (yz^5)z^{2s-8} \in I^2$, which proves that $ E\in I(\X_{\mathrm{off}})\cap I^2$.
\end{proof}

Recall that a finite set of $s\geq 3$ distinct points in $\pp^2$ are in \textit{general linear position }if no subset of three points lie on a line. 
\begin{Theorem}\label{s-3}
	Let $\X\subseteq \pp^2$ be an $(s-3)$-fold collinear configuration of $s\geq 9$ distinct points. Suppose, moreover, that the remaining three points are in general linear position. Then the Jacobian ideal of $I(\X)$ is not $\VaVa$-torsion-free.  
\end{Theorem}
\begin{proof}
	We compute a set of minimal generators of $I(\X)$ and this will work for any $s\geq 6$. Assume that $p_1,\ldots,p_{s-3}$ lie on the straight line $\{z=0\}$. Let $p_{s-2},\, p_{s-1},\, p_s$ the remaining points off the collinear line which are in general linear position. Let $l_{i,j}$ denote the linear form defining the straight line through points $p_i$ and $p_j$.  The defining ideal of $3$ points in general linear position is generated by three conics 
	\[q_1=l_{(s-2),s}l_{(s-1),s},\, q_2=l_{(s-2),s}l_{(s-2),(s-1)},\, q_3=l_{(s-1),s}l_{(s-2),(s-1)}. \]
	 Consider the polynomial $G:=l_{1,(s-2)}l_{2,(s-1)}\prod_{i=3}^{s-3}l_{is}$.  Setting 
	$$J:=(zq_1,\, zq_2,\, zq_3,\, G)\subset  R=k[x,y,z].$$
	By construction, the generators of $J$ vanishes on $\X$ and hence  $J\subseteq I(\X)$. We claim that $J=I(\X)$. Setting  $\mathfrak{a}:=(zq_1,\, zq_2,\, zq_3)$. Consider the short exact sequence 
	\[0\rar \frac{R}{(\mathfrak{a}:_R G)} (-(s-3))\rar \frac{R}{\mathfrak{a}}\rar\frac{R}{J} \rar 0. \]
	Direct inspection gives that $(\mathfrak{a}:_R G)=(z)$. One has 
	\begin{eqnarray}
	\nonumber \mathrm{HS}(R/J,t)&=&  \mathrm{HS}(R/\mathfrak{a},t)-\mathrm{HS}(\frac{R}{(z)}(-(s-3)),t))\\
	\nonumber &=&\dfrac{1+t+t^2-2t^3}{(1-t)^2}-\dfrac{t^{s-3}}{(1-t)^2}\\
	\nonumber &=&\dfrac{1+2t+3t^2+t^3+\ldots+t^{s-5}+t^{s-4}}{1-t}.
	\end{eqnarray}
	Next write $\X=\X_1\cup\X_2$, where $\X_1$ is the set of $(s-3)$ collinear points  and $\X_2$ is the set of $3$ points in general linear position. By Theorem~\ref{collinear-colinear-1}, we get $I(\X_1)=(z,f(x,y))$, where $f(x,y)$ is a reduced polynomial of degree $s-3$. We have 
	\[\mathrm{HS}(R/I(\X_1),t)=\dfrac{1+t+t^2+\ldots+t^{s-4}}{(1-t)}. \]   
	Also $R/I(\X_2)$ has the Hilbert series 
	\[\mathrm{HS}(R/I(\X_2),t)=\dfrac{1+2t}{(1-t)}.\]
	Since $\mathfrak{b}:=(I(\X_1),I(\X_2))=(x^2,xy,y^2,z)$, it follows that $R_2=\mathfrak{b}_2$, hence
	\[\mathrm{HS}(R/\mathfrak{b},t)=1+2t.\]
	Using the exact sequence (\ref{ExactSeq}) in this setup, we get  
	\begin{eqnarray}
	\nonumber \mathrm{HS}(R/I(\X),t)&=&\dfrac{1+t+t^2+\ldots+t^{s-4}}{(1-t)}+\dfrac{1+2t}{(1-t)}-(1+2t)\\
	\nonumber &=& \dfrac{1+2t+3t^2+t^3+\ldots+t^{s-5}+t^{s-4}}{1-t} 
	\end{eqnarray}
	which proves the claim. 
	
	We may assume that $p_i=[a_i:1:0]$ for $i=1,\ldots,s-3$ with $a_i\neq 0,1,-1$ and, by a projective transformation,  $p_{s-2}=[1:0:1],\, p_{s-1}=[0:1:1],\, p_s=[0:0:1]$. 
	In this setting, one gets $\mathfrak{a}=(zxy,\, x^2z-xz^2,\, y^2z-yz^2 )$ and  
	\[I(\X)=(\mathfrak{a},\, f(x,y)+z^{s-4}(ey-x)),\]
	where $f(x,y)=\prod_{i=1}^{s-3}(x-a_iy)$ and $e=\prod_{i=1}^{s-3}a_i$. The Jacobian matrix of $I(\X)$ is 
	\[\Theta=\begin{bmatrix}
	yz&xz&xy\\
	2xz-z^2&0&x^2-2xz\\
	0&2yz-z^2&y^2-2yz\\
	f_x-z^{s-4}& f_y+ez^{s-4}& (s-4)z^{s-5}(ey-x)
	\end{bmatrix}
	.\]
	Let  $I_1:=(\mathfrak{a},I_2(\Theta'))$, where $\Theta'$ is a   submatrix of $\Theta$ that we delete the last row. Thus the ideal  $I_1$ is minimally  generated by $\mathfrak{a}$ and the monomials $xz^3,yz^3,z^4$. We have the following relations
	\begin{eqnarray}
	\nonumber f(x,y)+z^{s-4}(ey-x) &=& \frac{1}{s-3}(xf_x+yf_y)+z^{s-4}(ey-x)\\
	\nonumber xyf_x&=&yG-yz^{s-4}(x+ey)-y^2f_y.
	\end{eqnarray}
	Note that by symmetry, the second relation holds for $xyf_y$. 
	Using the ideal $I_1$ and the above  relation, we conclude that 
	\[I=(\mathfrak{a},\, xz^3,\,yz^3,\, z^4,\, f,\, x^2f_x,\,x^2f_y,\,y^2f_x,\,y^2f_y).\]
	Setting $\mathfrak{b}:=(f,\, x^2f_x,\,x^2f_y,\,y^2f_x,\,y^2f_y)$. The ideal $\mathfrak{b}$ is zero dimensional homogeneous ideal. By the same argument as in the proof of Theorem~\ref{s-2points-collinear}, we conclude that $y^{2s-9}$ belongs to $\mathfrak{b}$ and the power is minimum. Also the polynomial $y^{2s-10}f-yz^{3s-14}\in I(\X)\cap I^2$, but not in $I(\X)I$ which prove that $\VaVa\neq 0$ in degree $2$. 
\end{proof}

\section{ Five  and six  points in $\pp^2$}\label{4-5-6points}
Let $\X$ be a set of $4$ points in $\pp^2$. There exists only three configurations on their geometry. More precisely, points are in general linear position, $\X$ is $3$-fold collinear configuration and $X$ is $4$-fold collinear configuration.  The Jacobian ideal of $I(\X)$ is not $\VaVa$-torsion-free. In fact, the general linear position  follows by \cite[Section 3.1]{AAR} and the other configurations follow by  Theorem~\ref{collinear-colinear-1}.  

\medskip

Let $\mathbb{X}=\{p_1,p_2,p_3,p_4,p_5\}$ be a set of $5$ points in $\mathbb{P}^2$. There exist only five configurations on their geometry: 
\begin{itemize}
	\item [{\rm (1)}] $\X$ is in general linear position.
	\item [{\rm (2)}] $\X$ is a $5$-fold collinear configuration.
	\item [{\rm (3)}] $\X$ is a $4$-fold collinear configuration.
	\item [{\rm (4)}] $\X$ is a $3$-fold collinear configuration.
	\item [{\rm (5)}] $\X$ is a $3$-fold collinear configuration such that the  straight line through the remaining two points intersect in a point of $\X$.
\end{itemize}
The following result characterizes $\VaVa$-torsion freeness of five points in their geometry.
\begin{Proposition}\label{5-points}
	Let $\mathbb{X}$ be a set of $5$ distinct points in $\mathbb{P}_k^2$. The Jacobian ideal of  $I(\X)$ is $\VaVa$- torsion-free if and only if $\X$ is  the configurations  {\rm (1),(4),(5)}.
	\end{Proposition}
\begin{proof}
By Theorem~\ref{collinear-colinear-1}, it is enough to show that for the configurations (1),(4),(5), $\VaVa_{J\subset I}=0$.

\underline{Configuration (1).} Note that the ideal of $4$ points in general linear position is generated by two conics. Then the ideal of $\X':=\X\setminus\{p_5\}$ is generated by conics say $q_1,q_2$ in $R=k[x,y,z]$. It is well known that there exists an unique conic $q$ passes through $5$ points. Since $I(\X)\subseteq I(\X')$, it follows that $q=aq_1+bq_2$ for certain uniquely determined nonzero scalars $a,b\in k$.  Since $I(\X)\subseteq I(p_5)=(\ell_1,\ell_2)$, we obtain that $q=L_1\ell_1+L_2\ell_2$, where $L_1,L_2$ are linear forms in $R$. We claim that 
\[I(\X)=(q,q_1\ell_1,q_1\ell_2). \]
Setting $J:=(q,q_1\ell_1,q_1\ell_2)$. One has $J\subseteq I(\X')\cap I(p_5)=I(\X)$. On the other hand, the ideal $J$ is generated by $2\times 2$ minors of the Hilbert-Burch matrix 
\[M=\begin{bmatrix}
q_1 & 0\\
L_1 & \ell_{2}\\
L_2 & \ell_1
\end{bmatrix}.
\]
Therefore,  the minimal free $R$-resolution of $R/J$ has the form  
\begin{equation*}\label{ResI}
0\lar R^2(-4)\stackrel{M}\lar R^2(-3)\oplus R(-2)\lar R\lar  R/ J\lar 0.
\end{equation*}
Thus the  Hilbert series of $R/J$ is
\[\mathrm{HS}(R/J,t)=\dfrac{1+2t+2t^2}{1-t}.\]
Since $I(\X)=I(\X')\cap I(p_5)$, one has a  short exact sequence 
\begin{equation}\label{MV}
0\lar R/I(\X)\lar R/I(\X')\bigoplus R/I(p_5)\lar R/( I(\X'), I(p_5))\lar 0.
\end{equation}
Direct inspection gives
\[R/I(p_5)\simeq k[z]\quad,\quad R/( I(\X'), I(p_5))\simeq k[z]/(z^2).\]
We obtain that 
\[\mathrm{HS}(R/I(\X),t)=\dfrac{1+2t+t^2}{(1-t)}+\dfrac{1}{1-t}-(1+t)=\dfrac{1+2t+2t^2}{(1-t)}.\]
Since $J\subseteq I(\X)$, it follows that $J=I(\X)$. 

By a projective transformation, we may assume that $p_1,p_2,p_3$ are  coordinate points, $p_4=[1:1:1]$ and $p_5=[a:b:1]$ where $a,b\neq 0,1$ and $a\neq b$. Then the defining ideal $I(\X)$ is minimally generated by:
\begin{eqnarray}
\nonumber q&=& (a-b)xy+(-ab+b)xz+(ab-a)yz,\\
\nonumber q_1\ell_1 &=& x^2y-(a+1)xyz+ayz^2,\\
\nonumber q_1\ell_2 &=& xy^2-bxyz-y^2z+byz^2.
\end{eqnarray}
A calculation yields that  the Jacobian ideal $I$ of $I(\X)$ minimally generated by $4$ cubics in $R/I(\X)$. A computation in \cite{Macaulay1} yields that the relation type number of the ideal Jacobian ideal  $I/I(\X)\subseteq R/I(\X)$ is two. Therefore, by \cite[Corollary 2.17]{AA}, it is enough to show that $J\cap I^2\subseteq JI$, which can be check easily.

\underline{Configurations (4),(5).} We apply Theorem~\ref{s-2points-collinear} -- to find the generators of the defining ideal -- and Proposition~\ref{ATFN}. Then $J\cap I^t=JI^{t-1}$ for $t\geq 5$. A calculation shows that $ J\cap I^m=JI^{m-1}$ for $m=2,3,4$, which complete the proof. 
\end{proof}

Now let $\mathbb{X}$ be a set of $6$ points in $\mathbb{P}^2$. There exists only eleven configurations in their geometry \cite{BE} {(we show these configurations schematically in figure~\ref{6points Configuration})}:
\begin{enumerate}\label{11configuration}
	\item [{\rm (1)}] $\mathbb{X}$ is in general linear position.
	\item  [{\rm (2)}] $\mathbb{X}$ is  a $6$-fold collinear configuration. 
	 \item  [{\rm (3)}] $\mathbb{X}$ is  a $5$-fold collinear configuration.
	 \item  [{\rm (4)}] $\mathbb{X}$ is  a $4$-fold collinear configuration.
	\item   [{\rm (5)}] $\mathbb{X}$ is a $4$-fold  collinear configuration such that  the straight line through the remaining two points intersect in a point of $\X$. 
	\item  [{\rm (6)}] $\mathbb{X}$ is a $3$-fold collinear configuration such that the remaining three points are in general linear position . 
	\item[{\rm (7)}] $\X$ is a $3$-fold collinear configuration such that the remaining points are collinear and  the  straight line through them  intersect in a point off $\X$.
	\item[{\rm (8)}] $\X$ is a $3$-fold collinear configuration such that  the straight line through two of  the  remaining  points  intersect in a point of $\X$.
	\item[{\rm (9)}] $\X$ is a $3$-fold collinear configuration such that the two straight line through of  the  remaining  points  intersect in  points of $\X$.
	\item[{\rm (10)}] $\X$ is a $3$-fold collinear configuration such that the  three  straight line through of  the remaining  points  intersect in  points of $\X$.
	\item  [{\rm (11)}]  $\X$   is on an irreducible conic.
\end{enumerate}
\begin{figure}[h]
	\centering
	\includegraphics[width=9cm, height=9cm]{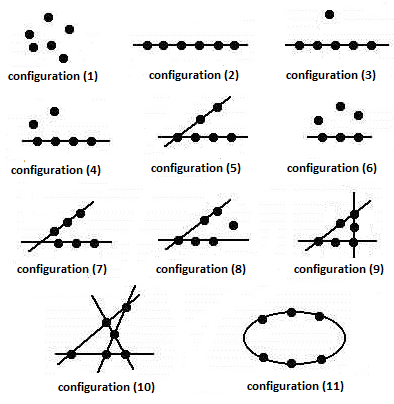}
	\caption{ Configurations of 6 points in $\pp^2$ }\label{6points Configuration}
\end{figure}
For $6=\binom{3+1}{2}$ points in general linear position we study the general case  of $s=\binom{d+1}{2}$ points with $d\geq 2$. 

Let $\X$ be a finite set of $s=\binom{d+1}{2}$ points in general linear position in $\pp^2$. Then  $\X$ has maximal Hilbert function by~\cite{Geramita-Marocia}, that is
\begin{equation}\label{HFG_Points}
{\rm HF}(R/I(\X),t)=\min\left\{s,\binom{t+2}{2}\right\}.
\end{equation} 
Let $\mathcal{N}$ denote the set of all monomials of degree $d$ in $R=k[x,y,z]$ expect the set of $d+4$ monomials 
\[x^{d}, x^{d-1}y, x^{d-2}y^2,\ldots, xy^{d-1}, y^d, x^{d-1}z, x^{d-2}yz, z^{d}.\]
Note that $|\mathcal{N}|=\binom{d+2}{d}-(d+4)=s-3$.
Let $h_i=\sum_{m\in \mathcal{N}}^{s-3}\alpha_i^{(j)}m$ denote a $k$-linear combination of all monomials in $\mathcal{N}$ with coefficient $\alpha_i^{(j)}\in k$ for $1\leq i\leq d+1$ and $1\leq j\leq s-3$. In~\cite{GO}, it is proved that the defining ideal of $s=\binom{d+1}{2}$ points in general linear position  is generated by $d+1$ forms of degree $d$. In the following we find these generators, explicitly. 
\begin{Proposition}\label{GLP(d+1)(2)}
	With assumption and notation as above, the defining ideal of $s=\binom{d+1}{2}$ points in general linear position in  $\pp^2$ is minimally  generated by form of degree $d$ of the form 
	\[ f_i=x^{d-i}y^i+\sum_{m\in \mathcal{N}}\alpha_i^{(j)}m,\,   f_d=x^{d-1}z+\sum_{m\in \mathcal{N}}\alpha_d^{(j)}m,\,  f_{d+1}=x^{d-2}yz+\sum_{m\in \mathcal{N}}\alpha_{d+1}^{(j)}m,  \]
where $1\leq j\leq s-3$	and the coefficient $\alpha_i^{(j)}$  are uniquely determined by the coordinate of the points. 
\end{Proposition} 
\begin{proof}
	We may assume that the points are the columns of the matrix 
	\[M:=\begin{bmatrix}
	1&0&0&1&a_1&a_2&\cdots&a_{s-4}\\
	0&1&0&1&b_1&b_2&\cdots&b_{s-4}\\
	0&0&1&1&1&1&\cdots&1
	\end{bmatrix}
	,\]
	Since the points are in general linear position, all $3\times 3$ minors of $M$ are non-zero. Consider the following matrices 
\[\tiny{
	A=\left[
	\begin{array}{ll|lll|lll|lllll}
	1&1&1&1&1&1&\cdots&1&1&1&\cdots&1&1\\
	a_1&b_1&a_1^2&a_1b_1&b_2^2&&\ldots&&a_1^{d-2}&a_1^{d-3}b_1&\cdots&a_1b_1^{d-3}&b_1^{d-2}\\
	a_2&b_2&a_2^2&a_2b_2&b_2^2&&\ldots&&a_2^{d-2}&a_2^{d-3}b_2&\cdots&a_2b_2^{d-3}&b_2^{d-2}\\
	\vdots&\vdots&\vdots&\vdots&\vdots&&\cdots&&\vdots&\vdots&\cdots&\vdots&\vdots\\
	a_{s-4}&b_{s-4}&a_{s-4}^2&a_{s-4}b_{s-4}&b_{s-4}^2&&\ldots&&a_{s-4}^{d-2}&a_{s-4}^{d-3}b_{s-4}&\cdots&a_{s-4}b_{s-4}^{d-3}&b_{s-4}^{d-2}\\
	\end{array}
	\right]
}\]
	and 
	\[B:=\begin{bmatrix}
	1&1&\cdots&1&1\\
	b_1^{d-1}&a_1b_1^{d-2}&\cdots&a_1^{d-4}b_1^{3}&a_1^{d-3}b_1^{2}\\
	\vdots&\vdots&\cdots&\vdots&\vdots\\
	b_{s-4}^{d-1}&a_{s-4}b_{s-4}^{d-2}&\cdots&a_{s-4}^{d-4}b_{s-4}^{3}&a_{s-4}^{d-3}b_{s-4}^{2}
	\end{bmatrix}
	.\]
		The matrix $A$ and $B$ are of size $(s-3)\times \dfrac{d^2-d-2}{2}$ and $(s-3)\times (d-2)$, respectively. Thus the concatenation of $A$ and $B$ is a square matrix of size $(s-3)\times (s-3)$.  To find $\alpha_i^{(j)}$ it is enough to solve the matrix equations
	\begin{equation}\label{Mat-equi1}
	\left[
	\begin{array}{l|l}
	A&B
	\end{array}
	\right]
	\begin{bmatrix}
	\alpha_i^{(1)}\\
	\alpha_i^{(2)}\\
	\vdots\\
	\alpha_i^{(s-4)}\\
	\alpha_i^{(s-3)}\\
	\end{bmatrix}=
	\begin{bmatrix}
	-1\\
	-a_1^{d-i}b_1^i\\
	\vdots\\
	-a_{s-2}^{d-i}b_{s-2}^i\\
	-a_{s-3}^{d-i}b_{s-3}^i\\
	-a_{s-4}^{d-i}b_{s-4}^i	
	\end{bmatrix}.
	\end{equation}
	for $i=1,\ldots, d-1$ and 
	\begin{equation}\label{Mat-equi2}
	\left[
	\begin{array}{l|l}
	A&B
	\end{array}
	\right]
	\begin{bmatrix}
	\alpha_d^{(1)}\\
	\alpha_d^{(2)}\\
	\vdots\\
	\alpha_d^{(s-4)}\\
	\alpha_d^{(s-3)}\\
	\end{bmatrix}=
	\begin{bmatrix}
	-1\\
	-a_1^{d-1}\\
	\vdots\\
	-a_{s-2}^{d-1}\\
	-a_{s-3}^{d-1}\\
	-a_{s-4}^{d-1}	
	\end{bmatrix} \quad , \quad 
	\left[
	\begin{array}{l|l}
	A&B
	\end{array}
	\right]
	\begin{bmatrix}
	\alpha_{d+1}^{(1)}\\
	\alpha_{d+1}^{(2)}\\
	\vdots\\
	\alpha_{d+1}^{(s-4)}\\
	\alpha_{d+1}^{(s-3)}\\
	\end{bmatrix}=
	\begin{bmatrix}
	-1\\
	-a_1^{d-2}b_1\\
	\vdots\\
	-a_{s-2}^{d-2}b_{s-2}\\
	-a_{s-3}^{d-2}b_{s-3}\\
	-a_{s-4}^{d-2}b_{s-4}	
	\end{bmatrix}.
	\end{equation}

	Since all $3\times 3$ minors  of $M$ are non-zero, the determinant of the matrix $\left[
	\begin{array}{l|l}
	A&B
	\end{array}
	\right]$ dose not vanish. Therefore, the systems~(\ref{Mat-equi1}) and~(\ref{Mat-equi2}) has unique solutions. Furthermore, by Cramer's rule, $\alpha_i^{(j)}\neq 0$. 
	Consider the ideal $J\subset R=k[x,y,z]$ generated by the form of degree $d$ as in the statement, where $\alpha_i^{(j)}$ are uniquely determined solution of (\ref{Mat-equi1}) and (\ref{Mat-equi2}). Thus the ideal $J$ vanishes on $\X$ and hence $J\subseteq I(\X)$.  We show that $J$ and $I(\X)$ have the same Hilbert function, hence must be equal. 
	
	We claim that the Gr\"obner basis of $J$ with respect to the deg-revlex term ordering with $x>y>z$ is the set 
	\[G\cup \{ y^dz+z\sum_{m\in \mathcal{N}}\beta^{(j)}m\}, \]
	where $G$ is the generating set of $J$ and $\beta^{(j)}$ is 
	a certain polynomial expression of the $\alpha$'s.  
	For this,  we consider the $S$-polynomials of elements in this set. First, we look at the $S$-polynomial of  $f_1$ and $f_2$
	\[\mathrm{S}(f_1,f_2)=yf_1-xf_2=y\sum_{m\in \mathcal{N}}\alpha_1^{(j)}m-x\sum_{m\in \mathcal{N}}\alpha_2^{(j)}m,\]
	which upon division by the generators of $J$ is reduces to $f_{d+2}:=y^dz+z\sum_{m\in \mathcal{N}}\beta^{(j)}m$, where $\beta^{(j)}$ is a certain polynomial like expression in the $\alpha$'s. 
Since the initial monomial $y^dz$ of $f_{d+2}$ is not divisible by the initial term of any generators of $J$ we add $f_{d+2}$ to the generating set of $J$.  By reducing the terms in  the $\mathrm{S}$-polynomial $S(f_i,f_j)$ with $1\leq i<j\leq d+2$
which are divisible by the initial term of $f_k\ (k=1,\ldots,d+2)$, we conclude that  $S(f_i,f_j)$ reduces to zero which prove the claim.

Thus, the following set of monomials is a minimal generating set for the initial ideal of $J$:
	\[x^{d-1}y, x^{d-2}y^2, \ldots, xy^{d-1},x^{d-1}z,x^{d-2}yz, y^dz.\] 
	Hence for $r\geq d$, 
	\[\dim_k(J)_r=\binom{r+2}{r}-\binom{d+1}{2}. \]
	Therefore for any $r\geq 0$,
	\[\mathrm{HF}(R/J,r)=\left\{
	\begin{array}{ll}
	\binom{d+1}{2},& \ r\geq d\\
	\binom{r+2}{r}, & 0\leq r\leq d-1.
	\end{array}
	\right.
	\]
	Therefore, $I(\X)=J$ follows by  (\ref{HFG_Points}). 
\end{proof}
\begin{Proposition}\label{VVGLP(d+1)2}
	Let $\X\subseteq \pp^2$ be a set of six points in general linear position. Then $I_2(\Theta)=(x,y,z)^4$, where $\Theta$ is the Jacobian matrix of $I(\X)$. In particular, the Jacobian ideal of $I(\X)$ is $\VaVa$-torsion-free.  
\end{Proposition}
{\begin{proof}
We may assume that the points are columns of the matrix 
\[M:=\begin{bmatrix}
1&0&0&1&a&c\\
0&1&0&1&b&d\\
0&0&1&1&1&1
\end{bmatrix}
.\]
Since the points are in general linear position, the following relations come out
\begin{equation}\label{deteq}
a,b,c,d\neq 0,1\quad, \quad a\neq b\neq c\neq d\quad,\quad a-b\neq c-d.
\end{equation}
By Proposition~\ref*{GLP(d+1)(2)}, the ideal $I(\X)$ is generated by the cubics:
\begin{eqnarray}
\nonumber f_1 &=& x^2y+\alpha_1^{(1)}xz^2+\alpha_1^{(2)}y^2z+\alpha_1^{(3)}yz^2\\
\nonumber f_2 &=& xy^2+\alpha_2^{(1)}xz^2+\alpha_2^{(2)}y^2z+\alpha_2^{(3)}yz^2\\
\nonumber f_3 &=& x^2z+\alpha_3^{(1)}xz^2+\alpha_3^{(2)}y^2z+\alpha_3^{(3)}yz^2\\
\nonumber f_4 &=& xyz+\alpha_4^{(1)}xz^2+\alpha_4^{(2)}y^2z+\alpha_4^{(3)}yz^2
\end{eqnarray}
where $\alpha_i^{(j)}$ for $i=1,2,3,4,\ j=1,2,3$ are solutions of the matrix equations 
\[
\begin{bmatrix}
1&1&1\\
a&b^2&b\\
c&d^2&d
\end{bmatrix}
\begin{bmatrix}
\alpha_1^{(1)}\\
\alpha_1^{(2)}\\
\alpha_1^{(3)}
\end{bmatrix}
=\begin{bmatrix}
-1\\
-a^{2}b\\
-c^{2}d
\end{bmatrix}\quad , \quad 
\begin{bmatrix}
1&1&1\\
a&b^2&b\\
c&d^2&d
\end{bmatrix}
\begin{bmatrix}
\alpha_2^{(1)}\\
\alpha_2^{(2)}\\
\alpha_2^{(3)}
\end{bmatrix}
=\begin{bmatrix}
-1\\
-ab^2\\
-cd^2
\end{bmatrix}
\]
\[
\begin{bmatrix}
1&1&1\\
a&b^2&b\\
c&d^2&d\end{bmatrix}
\begin{bmatrix}
\alpha_3^{(1)}\\
\alpha_3^{(2)}\\
\alpha_3^{(3)}
\end{bmatrix}
=\begin{bmatrix}
-1\\
-a^{2}\\
-c^{2}
\end{bmatrix}\quad , \quad 
\begin{bmatrix}
1&1&1\\
a&b^2&b\\
c&d^2&d\end{bmatrix}
\begin{bmatrix}
\alpha_4^{(1)}\\
\alpha_4^{(2)}\\
\alpha_4^{(3)}
\end{bmatrix}
=\begin{bmatrix}
-1\\
-ab\\
-cd
\end{bmatrix}
.\]
By relations (\ref{deteq}),  all $3\times 3$ minors of the following matrix is non-zero. 
\[
\begin{bmatrix}
1&1&1&-1&-1&-1&-1\\
a&b&b^2&-a^2b&-ab^2&-a^2&-ab\\
c&d&d^2&-c^2d&-cd^2&-c^2&-cd
\end{bmatrix}.
\]
 Therefore, by Cramer's rule,  $\alpha_i^{(j)}\neq 0$ which is a certain polynomials expression of the coordinate of last two points. 
The Jacobian matrix of $I(\X)$ is 
\[
\Theta=\begin{bmatrix}
2xy+\alpha_1^{(1)}z^2& x^2+2\alpha_1^{(2)}yz+\alpha_1^{(3)}z^2& \alpha_1^{(1)}xz+2\alpha_1^{(2)}y^2+2\alpha_1^{(3)}yz\\
y^2+\alpha_2^{(1)}z^2& 2xy+2\alpha_2^{(2)}yz+\alpha_2^{(3)}z^2& \alpha_2^{(1)}xz+2\alpha_2^{(2)}y^2+2\alpha_2^{(3)}yz\\
2xz+\alpha_3^{(1)}z^2& 2\alpha_3^{(2)}yz+\alpha_3^{(3)}z^2& x^2+2\alpha_3^{(1)}xz+\alpha_3^{(2)}y^2+2\alpha_3^{(3)}yz\\
yz+\alpha_4^{(1)}z^2& xz+2\alpha_4^{(2)}yz+\alpha_4^{(3)}z^2&  xy+2\alpha_4^{(1)}xz+\alpha_4^{(2)}y^2+2\alpha_4^{(3)}yz
\end{bmatrix}
.\]
Consider on the ring $R=k[x,y,z]$ the lex term ordering with $x>y>z$. Denote by $[i\ j\ | \ k \ l ]$ a $2\times 2$ minor of the Jacobian matrix $\Theta$, where $1\leq i<j\leq 4,\, 1\leq k<l\leq 3$. Setting
\[g_1=[1\ 3\ |\ 2\ 3]\ , \ g_2=[1\ 4\ |\ 2\ 3]\ , \  g_3=[3\ 4\ |\ 2\ 3]\ , \ g_4=[2\ 4\ |\ 2\ 3],\]
\[g_5=[3\ 4\ |\ 1\ 3]\ , \ g_6=[3\ 4\ |\ 1\ 2]\ , \  g_7=[2\ 4\ |\ 1\ 3]\ , \ g_8=[2\ 4\ |\ 1\ 2]. \]
 The initial terms of $g_1,\ldots,g_8$ are $x^4,\, x^3y,\, x^3z,\, x^2y^2,\, x^2yz,\, x^2z^2,\, xy^3,\, xy^2z$, respectively.
We have 
\[ [1\ 3\ |\ 1\ 2]=-2x^3z-\alpha_3^{(1)}x^2z^2+4\alpha_3^{(2)}xy^2z+(2\alpha_3^{(3)}-4\alpha_1^{(2)})xyz^2-2\alpha_1^{(3)}xz^3+\lambda_1yz^3+\lambda_2z^4,\]
where $\lambda_1=2\alpha_1^{(1)}\alpha_3^{(2)}-2\alpha_1^{(2)}\alpha_3^{(1)}$ and $\lambda_2=\alpha_1^{(1)}\alpha_3^{(3)}-\alpha_1^{(3)}\alpha_3^{(1)}$. If $(2\alpha_3^{(3)}-4\alpha_1^{(2)})\neq 0$, then we reduce  the terms $x^3z,\, x^2z^2,\, xy^2z$ by $g_1,\ldots,g_8$ which    gives the form, say $g_9$ and the latter has initial term $xyz^2$. If  $(2\alpha_3^{(3)}-4\alpha_1^{(2)})= 0$, then we choose another minor which contain the term $xyz^2$ and apply the same process as above. We make the same process to find the forms $g_{10},\ldots,g_{15}$ which have initial terms $$xz^3,\, y^4,\, y^3z,\, y^2z^2,\, yz^3,\, z^ 4,$$ respectively. 
Thus, the ideal generated by  $g_1,\ldots,g_{15}$ is contained  in the ideal $I_2(\Theta)$ and $g_i$'s are linearly independents over $k$ since its  coefficient matrix is upper triangular matrix which is non-singular by above construction.    
 Since $I_2(\Theta)\subseteq (x,y,z)^4$, it follows that  $I_2(\Theta)=(x,y,z)^4$. 
The second assertion follows by~\cite[Lemma 1.4]{AAR}. 
\end{proof}}
\begin{Remark}\rm
	By a similar argument as in Proposition~\ref{VVGLP(d+1)2}, we can show that $I_2(\Theta)=(x,y,z)^{2d-2}$ for $d=4,5$, which implies that the Jacobian ideal of the defining ideal of  a set of $10$ and $15$ points is $\VaVa$-torsion free. 
\end{Remark}
we derive the following 
\begin{Conjecure}
		Let $\X\subseteq \pp^2$ be a set of $s=\binom{d+1}{2}$ points in general linear position. If $d\geq 6$, then $(x,y,z)^{2d-2}\subseteq I$, where $I$ is the Jacobian ideal of $I(\X)$. 
\end{Conjecure}

Finally, we characterize $\VaVa$-torsion freeness of six points. 
\begin{Proposition}\label{6points}
	Let $\mathbb{X}$ be a set of $6$ points in $\mathbb{P}_k^2$. Then the Jacobian ideal of  $I(\X)$ is not $\VaVa$-torsion-free if and only if $\X$ is one of the configurations {\rm (2),(3),(7),(11)}.
\end{Proposition}
\begin{proof}
The configurations (2) and (3) follows by Theorem~\ref{collinear-colinear-1}.  The defining ideal $I(\X)=(f,g)$ of $\X$  for the configurations (7) and (11) minimally generated by a conic $f$ and a cubic $g$ (\cite{BE}). The Jacobian ideal $I$ of $J$ is zero dimensional ideal.
Thus $I\cap k[z]\neq 0$. We may assume that  $z^n\in I$ and $n$ is minimum with this property. Then $z^{n-1}f\in J\cap I^2$ and dose not belong to $JI$ by the minimality of $n$. Therefore, $\VaVa_{J\subset I}\neq 0$. 

Now we prove that  $\VaVa=0$ for the remaining seven configurations. 

\underline{Configuration (1)}. Follows by Proposition~\ref{VVGLP(d+1)2}.

\underline{Configuration (4),(5). }
we use Theorem~\ref{s-2points-collinear} -- to find the generators of the defining ideal -- and  Proposition(\ref{ATFN}). Then $J\cap I^t=JI^{t-1}$ for $t\geq 6$. A computation in~\cite{Macaulay1} shows that $ J\cap I^m=JI^{m-1}$ for $m=2,3,4,5$.  

\underline{Configuration (6)}. By Theorem~\ref{s-3}, the ideal $I(\X)$ is generated by  the $4$ forms of degree $3$.  The same argument as in Proposition~\ref{VVGLP(d+1)2}  implies that $I_2(\Theta)=(x,y,z)^4$, where $\Theta$ is the Jacobian matrix of $I(\X)$. Therefore, the assertion follows by~\cite[Lemma 1.4]{AAR}.

\underline{Configurations (8),(9),(10)}. 
By a projective transformation, we may assume that the  points in configurations (8) , (9) and (10) are the columns of the matrices 
\[
M_8:=\begin{bmatrix}
1&0&1&1&a&b\\
0&1&1&1&0&0\\
0&0&0&1&1&1\\
\end{bmatrix},\, \ 
M_9:=\begin{bmatrix}
1&0&1&0&a&0\\
0&1&1&0&0&b\\
0&0&0&1&1&1\\
\end{bmatrix},
\]
\[
M_{10}:=\begin{bmatrix}
1&0&1&a&b&a\\
0&1&1&0&0&a-b\\
0&0&0&1&1&1\\
\end{bmatrix}.
\]
where $a,b\in k$ with $a,b\neq 0,1$ and $a\neq b$.
By Theorem~\ref{s-3}, rather by its proof,mn the defining ideal  of these configurations are:
\begin{eqnarray}
\nonumber I(\X_8)&=&(x^2y-xy^2,\, x^2z-(a+b)xz^2+\lambda yz^2+abz^3,\, xyz-yz^2,\,y^2z-yz^2),\\
\nonumber I(\X_9)&=&(x^2y-xy^2,\, x^2z-axz^2,\,  xyz\ , \ y^2z-byz^2),\\
\nonumber I(\X_{10})&=&(x^2y-xy^2-abyz^2,\, x^2z-(a+b)xz^2+abz^3,\, xyz-ayz^2,\, y^2z+(b-a)yz^2),
\end{eqnarray}
where $\lambda=-ab+a+b-1$ in $I(\X_8)$. 

As the same argument for points in general linear position (Proposition~\ref{VVGLP(d+1)2}), we can find $15$ linearly independent forms of degree $4$ among $18$ non-zero minors of the Jacobian matrix of the defining ideals. Thus $I_2(\Theta)=(x,y,z)^4$ which proves the assertion.
\end{proof}

\end{document}